\documentclass[a4paper,12pt,oneside]{amsart}
\usepackage{amsmath}
\usepackage{amssymb}
\usepackage{amsthm}
\usepackage[english]{babel}
\usepackage{esint}
\usepackage{graphicx}
\usepackage{layout}
\usepackage{multido}
\usepackage{xcomment}
\usepackage{comment}
\usepackage[usenames,dvipsnames]{color}
\newcommand{\cd}{k}

\newcommand{\Img}{\mathrm{Im}}

\newcommand{\nint}{\int_{\RN}}

\newcommand{\no}[1]{\|#1\|}

\newcommand{\R}{\mathbb{R}}

\newcommand{\C}{\mathbb{C}}

\newcommand{\ra}{\rightarrow}

\newtheorem*{lemma*}{Lemma}
\ifx\theorem\undefined

\newtheorem{theorem}{Theorem}
\newtheorem*{theorem*}{Theorem}
\fi
\ifx\lemma\undefined
\newtheorem{lemma}{Lemma}[]
\fi

\ifx\prop\undefined
\newtheorem{prop}{Proposition}
\newtheorem*{prop*}{Proposition}
\fi

\newtheorem*{cor*}{Corollary}
\theoremstyle{definition}

\newtheorem{defn*}{Definition}
\theoremstyle{remark}

\newcommand{\mb}[1]{\mathbb{#1}}

\newcommand{\var}{\varepsilon}

\newcommand{\p}{\partial}

\newcommand{\om}{\omega}

\newcommand{\sg}{\sigma}

\newcommand{\N}{n}
\newcommand{\RN}{\R\sp\N}
\newcommand{\F}{F}
\newcommand{\kk}{j}
\newcommand{\pt}{\partial}
\newcommand{\D}{D}

\renewcommand{\om}{\mathbf{\omega}}

\newcommand{\M}{M}
\newcommand{\CC}{C}
\newcommand{\G}{G}
\newcommand{\HH}{H}

\newcommand{\EH}{\mathbf{E}}
\newcommand{\CH}{\mathbf{C}}

\newcommand{\z}{z}
\renewcommand{\sg}{\sigma}
\theoremstyle{theorem}
\usepackage[colorlinks,citecolor={magenta}]{hyperref}
\begin{document}
\title[Solutions with a small energy/charge ratio]%
{Existence of positive solutions to a semi-linear elliptic 
system with a small energy/charge ratio}
\author{Garrisi Daniele}
\address{Math Sci. Bldg Room \# 302, POSTECH, Hyoja-Dong, Nam-Gu, Pohang, 
Gyeongbuk, 790-784, Republic of Korea}
\curraddr{}
\email{garrisi@postech.ac.kr}
\thanks{This work was supported by Priority 
Research Centers Program through the National Research Foundation of Korea 
(NRF) funded by the Ministry of Education, Science and Technology 
(Grant 2010-0029638).}
\date{\today}
\begin{abstract}
We prove the existence of positive solutions to a system of $ \cd $ 
non-linear elliptic equations corresponding to standing-wave $ k $-uples
solutions to a system of non-linear Klein-Gordon equations. 
Our solutions are characterised by a small energy/charge ratio,
appropriately defined.
\end{abstract}
\maketitle
\section*{Introduction}
\thispagestyle{empty}
\noindent
Given the real numbers $ 0 < m_1\leq m_2\leq\dots\leq m_k $, 
we show the existence of solutions to the non-linear elliptic system 
\[
\label{eq:E}
\tag{E}
\begin{array}{c}
-\Delta u_j + (m_j \sp 2 - \om_j\sp 2) u_j
+ \pt_{z_j} \G(u) = 0,\quad 1\leq j\leq k\\\\
u_j > 0,\quad u_j\in H\sp 1 _r (\RN)
\end{array}
\]
which are critical points of the energy functional 
\begin{gather*}
E\colon\HH\times\Sigma\ra\R,\\
(u,\om)\mapsto \frac{1}{2}\sum_{j = 1}\sp k
\nint |\D u_j|\sp 2 + (m_j\sp 2 + \om_j\sp 2) u_j\sp 2 + 2k\sp {-1}
\G(u)
\end{gather*}
on the constraint
\begin{gather*}
\M_\sg := \{(u,\om)\in\HH\times\Sigma\,|\,\CC_j (u,\om) = \sg_j\}\\
\CC_j (u,\om) = \om_j \nint u_j\sp 2
\end{gather*}
for some $ \sg\in (0,+\infty)\sp k $. We used the notation
\[
\HH := H\sp 1 (\RN,\R\sp k),\quad\Sigma := [0,+\infty)\sp k.
\]
We also define
\[
H_r := H\sp 1 _r (\RN,\R\sp k),\quad M_\sg \sp r := M_\sg\cap H_r
\]
where, by definition, $ u\in H\sp 1 _r (\RN,\R\sp k ) $ if
$ u\in H\sp 1 (\RN,\R\sp k) $ and
\[
u\sp j (x) = u\sp j (y)\text{ if } |x| = |y|,\text{ a.e. }
\]
for every $ 1\leq j\leq k $. On the Hilbert spaces $ H $ and
$ H_r $, we consider the norm induced by the scalar product
\[
(u,v)_H := \sum_{j = 1} \sp k (u_j,v_j)_{H\sp 1}.
\]

Solutions to \eqref{eq:E} with the variational characterisation
above are interesting by several means: critical points of
$ E $ over $ \M_\sg $ correspond to standing-wave $ k $-uples solutions
to the system of non-linear Klein-Gordon equations
\[
\label{eq:k-NLKG}
\tag{$ k $-NLKG}
\pt_{tt} u_j - \Delta_x u_j + m_j\sp 2 u_j + \p_{z_j} \G(u) = 0,\quad
1\leq j\leq k
\]
through the map
\begin{equation}
\label{eq:sw}
(u,\om)\mapsto (e\sp{-i\om_1 t} u_1 (x),\dots,e\sp{-i\om_k t} u_k (x)).
\end{equation}
Secondly, if we denote $ H\sp 1 (\RN,\C\sp k)\times L\sp 2 (\RN,\C\sp k) $
by $ X $, on solutions to \eqref{eq:k-NLKG} the quantities
\begin{gather*}
\label{eq:energy}
\tag{Energy}
\EH\colon X\ra\R,\\ 
(\phi,\phi_t)\mapsto 
\frac{1}{2}\nint |\D\phi|\sp 2 + |\phi_t|\sp 2 + \frac{1}{2}\nint
\sum_{j = 1}\sp k \big(m_j\sp 2 \phi_j\sp 2 + 2k\sp{-1}\G(\phi)\big)
\end{gather*}
\begin{gather*}
\label{eq:charges}
\tag{Charges}
\CH_j \colon X\ra\R,\\
(\phi,\phi_t)\mapsto -\Img\nint \overline{\phi}_j \phi_t\sp j.
\end{gather*}
are constant (under the assumption $ \G(u) = \G(|u_1|,\dots,|u_k|) $)
and
\begin{align*}
E(u,\om) &= \EH(u_1,\dots,u_k,-i\om_1 u_1,\dots,-i\om_k u_k)\\
C(u,\om) &= \CH(u_1,\dots,u_k,-i\om_1 u_1,\dots,-i\om_k u_k).
\end{align*}
Such equalities turned out to be crucial to prove the orbital stability
of standing-wave solutions to the scalar NLKG in \cite{BBBM10}, and to 
a coupled NLKG in \cite{Gar11}. Finally, according to \cite{BBBS09},
solutions $ v $ to the scalar NLKG with initial datum $ \Phi\in X $ 
such that the energy/charge ratio
\[
\Lambda(\Phi) := \frac{\EH(\Phi)}{m\CH(\Phi)} < 1
\]
have a non-dispersive property. We do not address in this work
the orbital stability or dispersion. 
\vskip .5em
\noindent We use the notation
\[
m := m_1,\quad H_r \sp * := H_r\setminus 0,\quad\Sigma_* := \Sigma\setminus 0
\]
and assume that $ \G $ is continuously differentiable and
\[
\label{eq:A0}
\tag{$ A_0 $}
\G(\z) = \G(|\z_1|,\dots,|\z_k|); 
\]
\[
\label{eq:A1}
\tag{$ A_1 $} 
\F(\z) := \G(\z) + \frac{1}{2}\sum_{j = 1}\sp k m_j\sp 2 \z_j\sp 2\geq 0,\quad
\G(0) = 0;
\]
\[
\label{eq:A2}
\tag{$ A_2 $}
|\D\G(\z)|\leq c(|\z|\sp{p - 1} + |\z|\sp{q - 1}),\quad 
2 < p \leq q < \frac{2\N}{\N - 2};
\]
\[
\label{eq:A3}
\tag{$ A_3 $}
\alpha :=  \inf_{\z\in\Sigma_*} \frac{\F(\z)}{|\z|\sp 2} < \frac{m\sp 2}{2};
\]
for every $ 1\leq j \leq k $ 
\[
\label{eq:A4}
\tag{$ A_4 $}
\alpha_j := 
\inf_{\scalebox{0.57}{$ \displaystyle\sum_{h\neq j} \z_h \sp 2\neq 0 $}}
\frac{\F(\z)}{\sum_{h\neq j} \z_h \sp 2} > \alpha.
\]
Under the assumptions above, we can prove the following
\begin{theorem*}[Main]
\hypertarget{thm:main}{}
There exists an open subset $ \Omega\subset (0,+\infty)\sp k $ 
such that the infimum of $ E $ is achieved on $ \M_\sg \sp r $ for every
$ \sg\in\Omega $.
\end{theorem*}
The technique we use is similar to the one adopted in \cite{BF09} in
the scalar case $ k = 1 $. Therein it is showed that if a minimising 
sequence $ (u_n,\om_n) $ for $ E $ over $ M_\sg\sp r  $ is such that 
$ \om_n \ra \om < m $, then a subsequence of $ (u_n) $ converges
on $ H\sp 1 $. The existence of such sequences is provided by the
inequality
\begin{equation}
\label{eq:lambda}
\inf_{H_r \sp * \times\Sigma_*} \Lambda < 
\inf_{H_r \sp * \times\Sigma_* \sp m} \Lambda
\end{equation}
where 
\[
\Lambda(u,\om) := \frac{E(u,\om)}{C(u,\om)}
\]
and 
\[
\Sigma_* \sp m = \Sigma_* \cap\{z\geq m\}.
\]
In higher dimension, $ \Sigma_* \sp m $ should be replaced by
\[
\Sigma_* \sp{\mathbf{m}} := \bigcup_{j = 1} \sp k \Sigma_* \sp{m_j}
\]
where
\[
\Sigma_* \sp{m_j} = \{\z\in\Sigma_*\,|\,\z_j\geq m_j\}.
\]
A direct attempt to prove the inequality \eqref{eq:lambda} lead to minimise 
$ \Lambda(u,\cdot) $ over the set $ \Sigma_* \sp{\mathbf{m}} $,
whose boundary consists of $ 3\sp k - 1 $ pieces each of them 
leading to a different condition on the non-linear term
$ \F $. We believe that all these conditions include \eqref{eq:A4}. 

So, rather than proving \eqref{eq:lambda}, we show in 
Lemma~\hyperlink{lem:coercive}{Coercive} that when 
$ \Lambda(u_n,\om_n) $ converges to its infimum, each component of 
$ \om_n $ converges to $ \sqrt{2\alpha} < m $.
\section{Properties of the functional $ E $}
We recall some properties of the functional $ E $. We include the
proof of them only for the sake of completeness, as they are
similar to the scalar case \cite{BBBM10}.
\begin{prop}
\label{prop:properties}
Suppose that $ \G $ fullfils the assumptions \eqref{eq:A1} and
\eqref{eq:A2}. Then, $ E $ is continuously differentiable; if
$ \sg\in (0,+\infty)\sp k $, then $ E $ is coercive on $ \M_\sg $.
\end{prop}
\begin{proof}
The continuity and the differentiability of $ E $ follows from analogous
techniques used in theorems on bounded domains as 
\cite[Theorem~2.2 and 2.6, p.\,16,17]{AP93}.
For a detailed proof we also refer to \cite[Proposition~2]{Gar11}.\par
Let $ (u,\om)\in M_\sg $ and set $ E = E(u,\om) $. By \eqref{eq:A2}, we have
\begin{align}
\label{eq:prop:properties-1}
\om_i \leq\frac{2E}{\sg_i},\quad
\no{\D u}_{L\sp 2} \sp 2\leq 2E.
\end{align}
By \eqref{eq:A2} there exists $ \var > 0 $ such that
\begin{equation}
\label{eq:prop:properties-2}
\F(u)\geq m ^2 |u|^2 /4,\text{ if } |u|\leq\var.
\end{equation}
We have
\[
E\geq\int_{|u|\geq\var} \F(u) + \int_{|u| < \var} \F(u).
\]
From \eqref{eq:prop:properties-2}, it follows that
\begin{equation}
\label{eq:prop:properties-3}
\no{u}_{L^2 (|u| < \var)} ^2\leq 4E/m ^2.
\end{equation}
On the other hand, by the Sobolev inequality
\begin{equation}
\label{eq:prop:properties-4}
\begin{split}
\int_{|u|\geq\var} |u|\sp 2 =&\,\var\sp{2 - 2\sp *}\int_{|u|\geq\var}
\var\sp{2\sp * - 2} |u|\sp 2\leq 
\var\sp{2 - 2\sp *} \int_{|u|\geq\var} |u|\sp{2\sp *}\\
\leq&\,c\sp{2\sp *} \var\sp{2 - 2\sp*}\no{\D u}_{L\sp 2} \sp{2\sp *}
\end{split}
\end{equation}
where $ c $ is the constant in the proof of 
\cite[Th\'eor\`eme~IX.9,p.\,165]{Bre83}.
From \eqref{eq:prop:properties-3} and \eqref{eq:prop:properties-4}
\[
\no{u}\sp 2 _{L\sp 2}\leq 
\frac{4E}{m ^2} + 2 c\sp{2\sp *} \var\sp{2 - 2\sp*} E.
\]
Along with \eqref{eq:prop:properties-1}, we obtained that the
sub-levels of $ E $ are bounded, then $ E $ is coercive.
\end{proof}
\noindent 
Hereafter, we assume that $ \sg_j > 0 $ for every $ 1\leq j\leq k $.
\begin{prop}
\label{prop:palais-smale}
Let $ (u_n,\om_n)\subset M_\sg\sp r  $ be a Palais-Smale 
sequence and $ \om_n \ra\om $ such that $ \om_i < m_i $.
Then $ (u_n) $ has a converging subsequence.
\end{prop}
\begin{proof}
By Proposition~\ref{prop:properties}, $ (u_n) $
is bounded. Thus, by \cite[Theorem~A.I']{BL83-I}, we can suppose that 
\begin{equation}
\label{eq:prop:palais-smale-9}
u_n \sp j\rightharpoonup u_j\text{ in } H\sp 1 _r,\quad
u_n \sp j\ra u_j\text{ in } L\sp p \cap L\sp q 
\end{equation}
for every $ 1\leq j\leq k $.
Because $ (u_n,\om_n) $ is a Palais-Smale sequence, there are 
\[
(\lambda_n)\subset\R,\quad
(v_n,\eta_n)\subset H_r \sp *\times\R\sp k
\]
such that
\begin{equation}
\label{eq:prop:palais-smale-1}
\D E (u_n,\omega_n) = \sum_{j = 1}\sp k \lambda_n \sp\kk
\D C_j (u_n,\omega_n) + (v_n,\eta_n),\quad (v_n,\eta_n)\ra 0.
\end{equation}
We multiply \eqref{eq:prop:palais-smale-1} by $ (0,e_j)\in \{0\}\times\R\sp k $
and obtain
\[
\omega_n\sp\kk\no{u_n\sp\kk}_{L^2} ^2 = 
\lambda_n \sp\kk\no{u_n\sp\kk}_{L^2} ^2 + \eta_n \sp\kk
\]
whence
\begin{equation}
\label{ps:2}
\lambda_n \sp\kk  = \omega_n\sp\kk - \frac{\eta_n \sp\kk\omega_n \sp\kk}%
{\sg_j}.
\end{equation}
We multiply \eqref{eq:prop:palais-smale-1} by 
$ (\phi,0)\in H_r \times\{0\} $ and obtain
\[
\begin{split}
\sum_{j = 1}\sp k (\D u_n\sp j,\D\phi_j)_{L\sp 2} &+ 
m_j \sp 2 (u_n\sp j,\phi_j)_{L\sp 2} +
\nint \D\G(u_n)\cdot\phi \\
&+ 
\sum_{\kk = 1}\sp k (\omega_n \sp\kk)^2 (u_n \sp\kk ,\phi_\kk)_{L^2}
- 2\sum_{\kk = 1}\sp k\lambda_n \sp\kk  \omega_n \sp\kk
(u_n \sp\kk ,\phi_\kk)_{L^2} = (v_n,\phi)_H
\end{split}
\]
which, by (\ref{ps:2}), becomes
\begin{equation}
\label{eq:prop:palais-smale-3}
\begin{split}
\sum_{j = 1}\sp k (\D u_n\sp j,\D\phi_j)_{L\sp 2} &+ 
m_j \sp 2 (u_n\sp j,\phi_j)_{L\sp 2} +
\nint \D\G(u_n)\cdot\phi \\
&-  
\sum_{\kk = 1}\sp k (\om_n \sp\kk)\sp 2 (u_n \sp\kk,\phi_\kk)_{L^2} = 
(v_n,\phi)_H - 2\sum_{\kk = 1}\sp k \frac{\eta_n \sp\kk %
(\om_n \sp\kk)\sp 2}%
{\sg_j} (u_n \sp\kk ,\phi_\kk)_{L^2}.
\end{split}
\end{equation}
From \eqref{eq:A1}, \eqref{eq:prop:palais-smale-3} can
be written as
\begin{equation}
\label{ps:3} 
\begin{split}
&\sum_{j = 1}\sp k (\D u_n\sp j,\D\phi_j)_{L\sp 2} + 
(m_j \sp 2 - \om_j\sp 2)(u_n\sp j,\phi_j)_{L\sp 2}\\
=& 
(v_n,\phi)_H - 
\sum_{j = 1}\sp k \beta_n \sp j 
 (u_n\sp j,\phi_j)_{L\sp 2}\\
-& \nint (\D\G(u_n) - \D\G(u))\cdot(u_n - u)
\end{split}
\end{equation}
where
\begin{equation}
\label{eq:prop:palais-smale-11}
\beta_n\sp j := \left(\om_j\sp 2 - (\om_n \sp j)\sp 2 - 
\frac{2 \eta_n \sp j (\om_n \sp j)\sp 2}{\sg_j}\right)\ra 0,
\quad 1\leq j\leq k.
\end{equation}
Given a pair of integers $ (n,m) $, taking the difference of the equations, 
$ (\ref{ps:3}_n) $ and 
$ (\ref{ps:3}_m) $ with $ \phi = u_n - u_m $, we obtain
\begin{equation}
\label{eq:prop:palais-smale-5}
\begin{split}
&\sum_{\kk = 1}\sp k \|\D u_n\sp j - \D u_m\sp j\|_{L^2} ^2 + 
 (m_\kk ^2 - \om_\kk ^2 + \beta_n \sp j + \beta_m \sp j)
\|u_n \sp j - u_m\sp j\|_{L^2} ^2 \\
=& (v_n - v_m,u_n - u_m)_H - 
\nint\Big(\D\G(u_n) - \D\G(u_m)\Big)\cdot (u_n - u_m).
\end{split}
\end{equation}
Thus from the assumption $ \omega_j < m_j $ and 
\eqref{eq:prop:palais-smale-11}, there exists $ c_0 > 0 $ such that
\begin{equation}
  \label{eq:prop:palais-smale-8}
\begin{split}
\no{u_n - u_m}_H \sp 2&\leq
c_0\sum_{\kk = 1}\sp k \Big(\|\D u_n\sp j - \D u_m\sp j\|_{L^2} ^2
\\
& +  
 (m_\kk ^2 - \om_\kk ^2 + \beta_n \sp j + \beta_m\sp j)
\|u_n \sp j - u_m\sp j\|_{L^2} ^2\Big)
\end{split}
\end{equation}
and 
\begin{equation}
\label{eq:prop:palais-smale-7}
(v_n - v_m,u_n - u_m)_H \leq \no{u_n - u_m}_H (\gamma_n + \gamma_m)
\end{equation}
where
\begin{equation}
\label{eq:prop:palais-smale-12}
\gamma_n := \no{v_n}_H\ra 0.
\end{equation}
We have
\[
\begin{split}
&\left|\nint \Big(\D\G(u_n) - \D\G(u_m)\Big)\cdot (u_n - u_m)\right|\\
\leq&
\sum_{j = 1}\sp k \Big(\no{u_n \sp j - u_m\sp j}_{L\sp p} 
+ \no{u_n \sp j - u_m\sp j}_{L\sp q}\Big)
\end{split}
\]
it is convenient to estimate each of the two summand of the inequality above 
as follows: by 
\cite[Corollaire~IX.10,p.\,165]{Bre83}
\begin{equation}
\label{eq:prop:palais-smale-6}
\begin{split}
\no{u_n \sp j - u_m\sp j}_{L\sp p} &= 
\no{u_n \sp j - u_m\sp j}_{L\sp p}\sp{1/2}
\no{u_n \sp j - u_m\sp j}_{L\sp p}\sp{1/2}\\
&\leq
\no{u_n \sp j - u_m\sp j}_{L\sp p}\sp{1/2}
\no{u_n \sp j - u_m\sp j}_{H\sp 1}\sp{1/2}\\
&\leq \delta_{n,m} \sp p \no{u_n \sp j - u_m\sp j}_{H\sp 1}\sp{1/2}
\end{split}
\end{equation}
where
\[
\delta_{n,m}\sp p := \max_{1\leq j\leq k}
\no{u_n \sp j - u_m\sp j}_{L\sp p}\sp{1/2}
\]
is infinitesimal by \eqref{eq:prop:palais-smale-9}.
By the H\"older inequality, we have
\[
\sum_{j = 1}\sp k \no{u_n\sp j - u_m\sp j}_{H\sp 1}\sp{1/2}
\leq k\sp{4/3}\no{u_n - u_m}_H \sp{1/4}
\]
whence
\begin{equation}
\label{eq:prop:palais-smale-10}
\begin{split}
&\left|\nint \Big(\D\G(u_n) - \D\G(u_m)\Big)\cdot (u_n - u_m)\right|\\
\leq&
k\sp{4/3} (\delta_{n,m} \sp p + \delta_{n,m} \sp q) 
\no{u_n - u_m}_H \sp{1/4}.
\end{split}
\end{equation}
Now, putting together 
(\ref{eq:prop:palais-smale-8},\ref{eq:prop:palais-smale-7},
\ref{eq:prop:palais-smale-10}) we obtain
\[
(c_0)\sp{-1}
\no{u_n - u_m}_H \sp{7/8}
\leq c_1 (\gamma_n + \gamma_m) + \delta_n + \delta_m
\]
where
\[
c_1 = \sup_{n,m} \left(\sum_{j = 1}\sp k 
\no{u_n \sp j - u_m\sp j}_{H\sp 1}\sp 2\right)\sp{3/4}.
\]
Then each of $ (u_n \sp j) $ is a Cauchy sequence in $ H\sp 1 $ for
every $ 1\leq j\leq k $, thus converges to $ v_j\in H\sp 1 $. 
From \eqref{eq:prop:palais-smale-9}, $ v_j = u_j $, thus $ u_n\ra u $
in $ H_r $.
\end{proof}
\section{Properties of $ \Lambda $}
\noindent We define the following energy/charge ratio
\[
\Lambda(u,\om) := \frac{E(u,\om)}{\sum_{j = 1} \sp k C_j (u,\om)}
\]
and introduce the notation
\[
a(u) := \frac{1}{2}\nint |\D u|\sp 2 + \nint\F(u),\quad b_j (u) := 
\nint u_j \sp 2.
\]
If we fix $ u\in H_* $, we have the smooth function defined on
$ \Sigma_* $
\[
\Lambda(u,\cdot)\colon\Sigma\ra\R,\quad
\om\mapsto \Lambda(u,\om) = \frac{1}{2}\cdot 
\frac{2a(u) + \sum_{j = 1}\sp k b_j (u) \om_j \sp 2}%
{\sum_{j = 1} \sp k b_j (u)\om_j}
\]
It is not hard to check, arguing by induction on $ k $, 
that the following properties hold for $ \Lambda(u,\cdot) $:
\begin{enumerate}
\item is non-negative and achieves its infimum in a (unique) interior point
lying on the principal diagonal. We denote this point by $ \om(u) $ and
each of its components by $ \xi(u) $;
\item there holds
\[
\Lambda(u,\om(u)) = \xi(u),\quad \xi(u)\sp 2 = 
\frac{2a(u)}{\sum_{j = 1}\sp k b_j(u)}.
\]
\end{enumerate}
\begin{prop}
\label{prop:hylomorphy}
$ \inf_{H_*} \xi = \sqrt{2\alpha} $.
\end{prop}
\begin{proof}
That the right member is not greater than the left one, follows
from the definition of $ \alpha $. In fact,
\[
\xi(u)\sp 2 = \frac{\nint |\D u|\sp 2 + 2\nint \F(u)}%
{\sum_{j = 1}\sp k b_j (u)}\geq 
\frac{\nint |\D u|\sp 2 + 2\alpha \nint |u|\sp 2}%
{\nint |u|\sp 2}\geq 2\alpha,
\]
where in the last inequality we neglected the gradient terms.
In order to prove the opposite inequality, we define
\[
u_R (x) = 
\begin{cases}
\z &\text{ if } |x|\leq R\\
(1 + R - |x|) \z  & \text{ if } R\leq |x|\leq R + 1\\
0 & \text{ if } |x|\geq R + 1.
\end{cases}
\]
where $ \z\in\Sigma_* $ is an arbitrary point and
$ R > 0 $. We compute its gradient
\[
\D u_R \sp j (x) = 
\begin{cases}
0 & \text{ if } |x|\leq R \text{ or } |x| \geq R + 1\\
- \dfrac{\z_j x}{|x|} & \text{ otherwise. }
\end{cases}
\]
By standard computations, we have
\begin{gather*}
\no{u_R\sp j}_{L\sp 2} \sp 2 = \mu(B_1) R\sp\N \z_j\sp 2 + 
O(R\sp{\N - 1})\\
\nint \F(u_R) = \mu(B_1) R\sp\N \F(\z) + O(R\sp{\N - 1}),\\
\no{\D u_R\sp\kk}_{L\sp 2} \sp 2 = O(R\sp{\N - 2}),
\end{gather*}
where $ B_1 $ is the unit ball of $ \RN $ and $ \mu(B_1) $ is its
Lebesgue measure. Then,
\[
\xi(u_R)\sp 2 = \frac{2\mu(B_1) R\sp\N\F(\z) + o(R\sp\N)}%
{\mu(B_1) R\sp\N |\z|\sp 2 + o(R\sp\N)} = 
o(1) + \frac{2\F(\z)}{|\z|\sp 2}.
\]
Taking the limit as $ R\ra +\infty $, we obtain
\[
\inf_{H_*} \xi\sp 2\leq \frac{2\F(\z)}{|\z|\sp 2}
\]
for ever $ \z\in\Sigma_* $. Because $ \z $ was chosen arbitrarily,
we obtain the conclusion.
\end{proof}
Looking at the behaviour of $ \Lambda(u,\cdot) $,
one can easily deduce that sequences converging to the minimum value
converge to the minimum point. The next lemma exploits the uniform
behaviour of $ \Lambda $ on $ u $.
\begin{lemma*}[Coercive]
\hypertarget{lem:coercive}{}
For every $ \var > 0 $ there exists $ \eta $ such that 
\[
\Lambda(u,\om) < \sqrt{2\alpha} + \eta
\]
implies
\[
|\om_j - \sqrt{2\alpha}| < \var
\]
\end{lemma*}
\begin{proof}
For every $ 1\leq j\leq k $ and $ u\in H_* $, we define
\[
B_j (u) = \frac{b_j (u)}{\sum_{j = 1} \sp k b_j (u)}.
\]
We divide the proof in three steps.

\noindent\textsl{Step 1}. We show that if $ k\geq 2 $ and
$ \eta $ is small enough, there exists $ \delta_0\in (0,1) $ such that 
\begin{equation}
\label{eq:Bj}
B_j (u)\in (\delta_0,1 - \delta_0).
\end{equation}
It is useful to define $ \alpha_* := \min\{\alpha_j\,|\,1\leq j\leq k\} $.
Due to \eqref{eq:A4} we have $ \alpha < \alpha_* $.
By property (ii) of $ \Lambda $
\begin{equation}
\label{eq:Bj-1}
\sqrt{2\alpha} + \eta > \Lambda(u,\om)\geq \xi(u);
\end{equation}
we fix $ 1\leq j\leq k $. We have
\[
\begin{split}
\xi(u)\sp 2 &= \frac{\no{\D u}_{L\sp 2} \sp 2 + 2\nint \F(u)}%
{\sum_{j = 1} \sp k b_j (u)} \\
&= \frac{\no{\D u}_{L\sp 2} \sp 2 + 2\nint \F(u)}%
{\sum_{j\neq s} b_j (u)}\cdot\frac{1}{1 + B_j (u)}
\geq\frac{2\alpha_j}{1 + B_j (u)}
\end{split}
\]
where in the last inequality we neglected the gradient terms and used the 
notation of the assumption \eqref{eq:A4}.
From \eqref{eq:Bj-1} and the inequality above, we obtain
\[
\sqrt{2\alpha} + \eta > \frac{\sqrt{2\alpha_j}}{\sqrt{1 + B_j (u)}}
\]
whence
\begin{equation}
\label{eq:Bj-3}
B_j (u) > \frac{2\alpha_j}{(\sqrt{2\alpha} + \eta)\sp 2} - 1 \geq
\frac{2\alpha_*}{(\sqrt{2\alpha} + \eta)\sp 2} - 1 =:  \delta_0.
\end{equation}
Thus, if $ \delta_0 > 0 $, the obtain a bound from below for $ B_j (u) $.
Thus, we require
\begin{equation}
\label{eq:Bj-2}
\eta < \sqrt{2\alpha_*} - \sqrt{2\alpha}
\end{equation}
which gives $ B_j (u) > \delta_0 $ for every $ 1\leq j\leq k $. 
Because 
\[
\sum_{j = 1} \sp k B_j (u) = 1
\]
it follows that
\[
B_j (u) = 1 - \sum_{h\neq j} B_h (u)\leq 1 - (k - 1)\delta_0\leq 1 - \delta_0.
\]
\textsl{Step 2.}
If $ \Lambda(u,\om) < \sqrt{2\alpha} + \eta $, then $ \om $
is bounded from above. If $ \eta $ is chosen as in \eqref{eq:Bj-2}
and $ k\geq 2 $ then 
\[
\Lambda(u,\om)\geq\frac{\sum_{j = 1} \sp k B_j \om_j\sp 2}%
{2\sum_{j = 1} \sp k B_j \om_j}\geq \frac{\delta_0}{2(1 - \delta_0)}\cdot
\frac{\sum_j \om_j \sp 2}{\sum_j \om_j}.
\]
Thus,
\[
\sum_{j = 1}\sp k\om_j \sp 2\leq 2 C_0
\cdot\sum_{j = 1} \sp k \om_j
\]
where
\[
C_0 := \frac{(\sqrt{2\alpha} + \eta)(1 - \delta_0)}{\delta_0}
\]
Thus,
\begin{equation}
\label{eq:om}
\om_j < C_0 (1 + \sqrt{k}),\quad 1\leq j\leq k.
\end{equation}
When $ k = 1 $,
\[
\sqrt{2\alpha} + \eta > \Lambda(u,\om)\geq\om/2 
\]
thus,
\begin{equation}
\label{eq:om-1}
\om < 2(\sqrt{2\alpha} + \eta).
\end{equation}
\textsl{Step 3.}
We conclude the proof of the lemma. When $ k\geq 2 $,
\[
\begin{split}
\eta \geq &\Lambda(u,\om) - \Lambda(u,\om(u)) = \Lambda(u,\om) - \xi(u) \\
=& \frac{1}{2}
\left(\frac{\xi\sp 2 + \sum_{j = 1} \sp k B_j \om_j \sp 2 - 2
\sum_{j = 1}\sp k B_j \om_j \xi}{\sum_{j = 1}\sp k B_j \om_j}\right)\\
=& \frac{1}{2}\frac{\sum_{j = 1} \sp k B_j (\om_j - \xi)\sp 2}%
{\sum_{j = 1}\sp k B_j \om_j} 
= \frac{1}{2}
\sum_{j = 1} \sp k \left(\frac{B_j}{\sum_{j = 1}\sp k B_j \om_j}\right)\cdot
(\om_j - \xi)\sp 2\\
\geq&\frac{\delta_0}{2(1 - \delta_0) C_0 (\sqrt{k} + 1)}
\sum_{j = 1} \sp k (\om_j - \xi)\sp 2
\end{split}
\]
the last inequality follows from the bounds on 
$ \om $ \eqref{eq:om} and on $ B_j $ from \textsl{Step 1}
and \textsl{Step 2}. Thus,
\[
\frac{2\eta (1 - \delta_0)\sp 2 (\sqrt{k} + 1)(\sqrt{2\alpha} + \eta)}%
{\delta_0 \sp 2} > 
(\om_j - \xi)\sp 2.
\]
Because $ \xi < \sqrt{2\alpha} + \eta $, 
\begin{equation}
\label{eq:Bj-4}
|\om_j - \sqrt{2\alpha}| < \sqrt{\eta} \left(\sqrt{\eta} + 
\frac{1 - \delta_0}{\delta_0}\cdot
\left(2(\sqrt{2\alpha} + \eta)(\sqrt{k} + 1)\right)\sp{1/2}\right)
\end{equation}
for every $ 1\leq j\leq k $.
Because the term on the right member of the inequality above is
$ O(\sqrt{\eta}) $, the proof is complete when $ k\geq 2 $.
When $ k = 1 $, by \eqref{eq:om-1}
\[
\eta > \Lambda(u,\om) - \xi(u) = \frac{1}{2\om} (\om - \xi)\sp 2\geq
\frac{1}{4(\sqrt{2\alpha} + \eta)} (\om - \xi)\sp 2
\]
then
\[
|\om - \xi| < 2\left(\eta(\sqrt{2\alpha} + \eta)\right)\sp{1/2}
\]
whence
\begin{equation}
\label{eq:Bj-5}
|\om - \sqrt{2\alpha}| < \sqrt{\eta}\left(\sqrt{\eta} + 
2\left(\sqrt{2\alpha} + \eta\right)\sp{1/2} \right)
\end{equation}
\end{proof}
\def\proofname{Proof of the Theorem \hyperlink{thm:main}{Main}}
\begin{proof}
Let $ (u',\om') $ be such that 
\[
\Lambda(u',\om') < \sqrt{2\alpha} + \eta
\]
where $ \eta $ is chosen in such a way that the right term in \eqref{eq:Bj-4}
(for $ k\geq 2 $) or \eqref{eq:Bj-5} (when $ k = 1 $)
is not greater than
\[
\frac{1}{2}(m - \sqrt{2\alpha}).
\]
We define 
\[
\sg_j := \om' _j \nint (u_j ')\sp 2.
\]
Clearly $ (u',\om')\in M_\sg\sp r  $. Now, let us take a minimising sequence
$ (u_n,\om_n) $ of $ E $ over $ M_\sg\sp r  $. By the
Ekeland's theorem \cite[Theorem~5.1,p.\,48]{Str90}, we can suppose
that $ (u_n,\om_n) $ is a Palais-Smale sequence.
Then, there exists $ n_0\in\mb{N} $ such that 
\[
\Lambda(u_n,\om_n) \leq \Lambda(u',\om') = 
\Lambda(u',\om') < \sqrt{2\alpha} + \eta.
\]
if $ n\geq n_0 $. Thus
\[
\Lambda(u_n,\om_n) < \sqrt{2\alpha} + \eta,\quad n\geq n_0.
\]
By the preceding lemma, we have 
\[
|\om_n \sp j - \sqrt{2\alpha}| < \frac{1}{2}(m - \sqrt{2\alpha});
\]
up extract a subsequence from $ (\om_n \sp j) $, we can suppose
that each of the $ (\om_n \sp j) $ converge to some $ \om_j $.
Therefore
\[
m - \om_j  = m - \sqrt{2\alpha} + \sqrt{2\alpha} - \om_j \geq
\frac{1}{2}(m - \sqrt{2\alpha}) > 0.
\]
By Proposition~\ref{prop:palais-smale}, we obtain that $ E $ achieves
its infimum on $ M_\sigma $. Finally, we observe that the subset of
$ (0,+\infty)\sp k $ 
\[
\Omega := 
\bigg\{%
\sigma\in (0,+\infty)\sp k\,|\,
\frac{I(\sigma)}{\sum_{j = 1} \sp k \sg_j} < \sqrt{2\alpha} + \eta
\bigg\}
\]
is open. In fact, let $ \sg_0\in\Omega $ and $ (u_0,\om_0) $ be 
a minimiser of $ E $ over $ M_{\sg_0} $. Thus,
\[
\Lambda(u_0,\om_0) < \sqrt{2\alpha} + \eta.
\]
Given an arbitrary $ \sg $, we define
\[
\om_\sg\sp j := \frac{\om_0 \sp j \sg_j}{\sg_0 \sp j}.
\]
Using the continuity of $ \Lambda $ on $ \om $, it can be showed that
\[
\Lambda(u_0,\om_\sg) = \Lambda(u_0,\om_0) + O(|\sg - \sg_0|).
\]
Thus, if $ |\sg - \sg_0| $ is small enough,
\[
\Lambda(u,\om_\sg) < \sqrt{2\alpha} + \eta
\]
which concludes the proof.
\end{proof}
\def\proofname{Proof}
\begin{cor*}
There exists $ \eta_0 $ such that, for every $ \eta < \eta_0 $ there
exists $ (u_\eta,\om_\eta) $ such that $ u_\eta $ is a solution to 
\eqref{eq:E}
\begin{gather*}
\Lambda(u_\eta,\om_\eta) < \sqrt{2\alpha} + \eta,\quad
\om_\eta \sp j - \sqrt{2\alpha} < \eta
\end{gather*}
for every $ 1\leq j\leq k $.
\end{cor*}
\begin{proof}
The existence of $ (u_\eta,\om_\eta) $ follow from Theorem~\hyperlink{thm:main}{Main}. All we need to prove is that $ u_\eta > 0 $ and solves the
elliptic system in \eqref{eq:E}.
So, let $ \sg\in (0,\infty)\sp k $ be as in 
Theorem~\hyperlink{thm:main}{Main} and $ (u,\om)\in M_\sg\sp r  $ a
minimiser of $ E $ over $ M_\sg\sp r  $. From \eqref{eq:A0}, 
\[
(v,\om) := (|u_1|,\dots,|u_k|,\om)
\]
is also a minimiser of $ E $ over $ M_\sg\sp r $ and, thus, a constrained 
critical point. 
There is a natural action
of the orthogonal group $ O(\N,\R) $ on $ H\sp 1 (\RN,\R\sp k) $
\begin{gather*}
O(\N)\times H\sp 1 (\R\sp\N,\R\sp k)\times [0,+\infty)\sp k 
\ra H\sp 1 (\R\sp\N,\R\sp k)\times [0,+\infty)\sp k\\
(G,u,\om)\mapsto G\cdot (u,\om) := (u(Gx),\om)
\end{gather*}
this action restricts to $ M_\sg $ and the set of fixed point is
$ M_\sg\sp r  $. Moreover, $ E $ is invariant for the action
\[
E(u,\om) = E(u(Gx),\om).
\]
By the symmetric criticality principle \cite[\S 0]{Pal79}, $ (u,\om) $
is a critical point of $ E $ over $ M_\sg $. Thus, each 
of the equations in \eqref{eq:E} can be written as
\begin{equation}
\label{eq:E-2}
-\Delta v_j + c_j (x) v_j = 0
\end{equation}
where
\begin{equation}
\label{eq:maxp}
c_j (x) = 
\begin{cases}
m_j\sp 2 - \om_j\sp 2 + \frac{\pt_{z_j} \G(v)}{v_j} 
&\text{ if } v_j (x)\neq 0\\
m_j\sp 2 - \om_j\sp 2 &\text{ if } v_j (x) = 0.
\end{cases}
\end{equation}
From \eqref{eq:A2}
\begin{equation}
\label{eq:maxp-1}
|c_j (x)|\leq m_j\sp 2 - \om_j\sp 2 + 
c\,(|v_j|\sp{p - 2} + |v_j|\sp{q - 2}).
\end{equation}
Thus, for every bounded domain $ V\subset \RN $, $ c_j\in L\sp\infty (V) $,
because $ v_j $ is continuous. Then, we can apply the maximum
principle to the elliptic equation \eqref{eq:E-2} (for example,
\cite[Lemma~1,p.\,556]{Eva10}) and conclude that $ v_j > 0 $ on $ V $.
Because this holds for every $ V $, $ v_j > 0 $ on $ \RN $. Hence
$ u_\eta $ has a sign for every $ \eta $. Up to adjusting the signs
of $ u_\eta\sp j $, $ (u,\om) $ is the sought solution to \eqref{eq:E}.
\end{proof}
\noindent Some remarks are in order.
\subsubsection*{Concentration of minimising sequences}
If we add the requirement
\[
\label{eq:A5}
\tag{$ A_5 $}
\nint \F(u_1 \sp *,\dots,u_k \sp *)\leq \nint \F(u)
\]
where $ u_j \sp * $ denotes the decreasing rearrangment of $ u_j $,
then minimisers of $ E $ over 
$ M_\sg \sp r $ are minimisers of $ E $ over $ M_\sg $. We define
\[
I(\sg) := \inf_{M_\sg} E.
\]
Moreover, if for every minimiser $ (u,\om) $ there holds
\[
\label{eq:A6}
\tag{$ A_6 $}
\varlimsup
E(u_1 (\cdot + y_n\sp 1),u_2 (\cdot + y_n \sp 2),\dots,
u_k (\cdot + y_n \sp k),\om) > E(u,\om)
\]
if $ |y_n \sp j - y_n \sp h| $ is not bounded for some $ j\neq h $,
then it is natural to expect the \textsl{sub-additivity} property
of $ I $, that is
\[
I(\sg) < I(\sg') + I(\sg - \sg')
\]
for every $ \sg' $ such that $ \sg'\neq\sg $ and $ \sg'_j\leq\sg $
for every $ 1\leq j\leq k $. Thus, by means of the concentration-compactness 
Lemma, it would follow that a minimising sequence exhibits a
concentration behaviour.
\subsubsection*{Some example of non-linearity}
It might be surprising the fact that in our solutions
all the frequencies tend to converge in the interval $ (\sqrt{2\alpha},m) $
regardless of the relations between $ m_j $ and $ m_h $ for $ j\neq h $.
This follows from the assumption \eqref{eq:A4}: when the non-linearity
$ \G $ does not have coupling terms, that is
\begin{equation}
\label{eq:uncoupled}
\G(\z) = \G_1 (\z_1) + \dots + \G_k (\z_k)
\end{equation}
the system \eqref{eq:E} reduces to $ k $ scalar elliptic equations
\[
-\Delta u_j + (m_j\sp 2 - \om_j\sp 2) u_j + \G_j '(u_j) = 0
\]
each of them can be solved separately as in \cite{BF09} or \cite{BBBM10}
in order to obtain positive solutions. By the Derrick-Pohozaev identity and 
the maximum principle it follows
\[
m_j > \om_j > \sqrt{2\alpha_j},\quad 1\leq j\leq k.
\]
So, if $ \G $ is as in \eqref{eq:uncoupled}, the frequencies $ \om_j $
have a different behaviour from the one proved in 
Theorem~\hyperlink{thm:main}{Main}, where
\[
\sqrt{2\alpha} < \om_j < m \leq m_j,\quad 1\leq j\leq k.
\]
In fact, a non-linearity as in \eqref{eq:uncoupled} does not satisfy the 
assumption \eqref{eq:A4}: given $ \z\neq (0,\dots,0) $, we have
\[
\frac{\F(\z)}{|\z|\sp 2}\geq\frac{\sum_{j = 1}\sp k \alpha_j |\z_j|\sp 2}%
{|\z|\sp 2}\geq\min_{1\leq j\leq k} \alpha_j.
\]
Also, it is more simple to treat each equation of 
the case \eqref{eq:uncoupled} separately, using the result of \cite{BBBM10}
or the theorem when $ k = 1 $.

Some non-linearities $ \G $ satisfying assumptions (\ref{eq:A1}--\ref{eq:A4})
are given by
\begin{gather*}
\G(\z) = - \z_1 \sp p \z_2 \sp p + |\z|\sp q,\quad \z\in\Sigma\\
\G(\z) := \G(|\z_1|,|\z_2|)
\end{gather*}
when $ k = 2,N = 3 $ and 
\[
1 < p,\ 2p < q < 5.
\]
When $ k = 3,N = 3 $, we can define
\begin{gather*}
\G(\z) = - (\z_1 \z_2)\sp{p_1} - (\z_2 \z_3)\sp{p_2}
- (\z_1 \z_3)\sp{p_3} 
- (\z_1 \z_2 \z_3)\sp{p_4} + |\z|\sp q,\quad \z\in\Sigma\\
\G(\z) := \G(|\z_1|,|\z_2|,|\z_3|).
\end{gather*}
and
\begin{gather*}
2 < 2p_i < q < 5,\text{ for } 1\leq i\leq 3\\
3 < 3p_4 < q.
\end{gather*}
\subsection*{Acknowledgments.}
To professor Vieri Benci and professor Jaeyoung Byeon for their helpful 
suggestions, as well as professors Pietro Majer, Claudio Bonanno and 
Jacopo Bellazzini. 
\nocite{Mus91}
\def\cprime{$'$} \def\cprime{$'$} \def\cprime{$'$} \def\cprime{$'$}
  \def\cprime{$'$} \def\cprime{$'$} \def\cprime{$'$} \def\cprime{$'$}
  \def\cprime{$'$} \def\polhk#1{\setbox0=\hbox{#1}{\ooalign{\hidewidth
  \lower1.5ex\hbox{`}\hidewidth\crcr\unhbox0}}}

\end{document}